\documentclass[reqno,a4paper]{amsart}
\usepackage{amssymb,graphicx}
\usepackage[colorlinks=true,hypertex]{hyperref}
\allowdisplaybreaks[4]
\theoremstyle{plain}
\newtheorem{thm}{Theorem}

\theoremstyle{remark}
\newtheorem{rem}{Remark}
\DeclareMathOperator{\td}{d\mspace{-2mu}}
\date{Commenced on 27 January 2009 and completed on 13 February 2009 in Melbourne}
\date{}

\begin{document}

\title[Sharpening and generalizations of Carlson's double inequality]
{Sharpening and generalizations of Carlson's double inequality for the arc cosine function}

\author[B.-N. Guo]{Bai-Ni Guo}
\address[B.-N. Guo]{School of Mathematics and Informatics, Henan Polytechnic University, Jiaozuo City, Henan Province, 454010, China}
\email{\href{mailto: B.-N. Guo <bai.ni.guo@gmail.com>}{bai.ni.guo@gmail.com}, \href{mailto: B.-N. Guo <bai.ni.guo@hotmail.com>}{bai.ni.guo@hotmail.com}}
\urladdr{\url{http://guobaini.spaces.live.com}}

\author[F. Qi]{Feng Qi}
\address[F. Qi]{Research Institute of Mathematical Inequality Theory, Henan Polytechnic University, Jiaozuo City, Henan Province, 454010, China}
\email{\href{mailto: F. Qi <qifeng618@gmail.com>}{qifeng618@gmail.com}, \href{mailto: F. Qi <qifeng618@hotmail.com>}{qifeng618@hotmail.com}, \href{mailto: F. Qi <qifeng618@qq.com>}{qifeng618@qq.com}}
\urladdr{\url{http://qifeng618.spaces.live.com}}

\begin{abstract}
In this paper, we sharpen and generalize Carlson's double inequality for the arc cosine function.
\end{abstract}

\keywords{sharpening, generalization, Carlson's double inequality, arc cosine function, monotonicity}

\subjclass[2000]{Primary 33B10; Secondary 26D05}

\thanks{The second author was partially supported by the China Scholarship Council}

\thanks{This paper was typeset using \AmS-\LaTeX}

\maketitle

\section{Introduction and main results}

In~\cite[p.~700, (1.14)]{Carlson-pams-ellip} and~\cite[p.~246, 3.4.30]{mit}, it was listed that
\begin{equation}\label{Carlson's-ineq-arccos}
\frac{6(1-x)^{1/2}}{2\sqrt2\,+(1+x)^{1/2}}<\arccos x<\frac{\sqrt[3]4\,(1-x)^{1/2}}{(1+x)^{1/6}}, \quad 0\le x<1.
\end{equation}
\par
The first aim of this paper is to sharpen and generalize the right-hand side inequality in~\eqref{Carlson's-ineq-arccos} as follows.

\begin{thm}\label{Carlson-Arccos-thm-1}
For real numbers $a$ and $b$, let
\begin{equation}
f_{a,b}(x)=\frac{(1+x)^b}{(1-x)^a}\arccos x,\quad x\in(0,1).
\end{equation}
\begin{enumerate}
\item
If and only if
\begin{equation}\label{Carlson-Arccos-thm-1-a2}
(a,b)\in\biggl\{b\le\frac2\pi-a\biggr\}\cap\biggl\{a\le\frac12\biggr\},
\end{equation}
the function $f_{a,b}(x)$ is strictly decreasing;
\item
If
\begin{align}\label{Carlson-Arccos-thm-1-a1}
(a,b)&\in \biggl\{\frac2\pi-a\le b\le a-\frac4{\pi^2}\biggr\}\cup\biggl\{\frac12\le a\le b +\frac13\biggr\}\\
&\quad\cup\biggl\{\frac13<a-b<\frac4{\pi^2},a+b\ge\frac{2(a-b)^{3/2}}{\sqrt{4(a-b)-1}\,}\biggr\},
\end{align}
the function $f_{a,b}(x)$ is strictly increasing;
\item
If
\begin{equation}\label{Carlson-Arccos-thm-1-a3}
(a,b)\in\biggl\{\frac13<a-b<\frac4{\pi^2}\biggr\}\cap\biggl\{\frac2\pi-b<a\le\frac12\biggr\},
\end{equation}
the function $f_{a,b}(x)$ has a unique maximum;
\item
If
\begin{equation}\label{Carlson-Arccos-thm-1-a4}
(a,b)\in\biggl\{\frac13<a-b<\frac4{\pi^2}\biggr\}\cap\biggl\{\frac12< a\le\frac2\pi-b\biggr\},
\end{equation}
the function $f_{a,b}(x)$ has a unique minimum;
\item
If
\begin{equation}\label{Carlson-Arccos-thm-1-a5}
(a,b)\in\biggl\{\frac13<a-b<\frac4{\pi^2}\biggr\}\cap \biggl\{\frac2\pi<a+b<\frac{2(a-b)^{3/2}}{\sqrt{4(a-b)-1}\,}\biggr\} \cap\biggl\{a>\frac12\biggr\},
\end{equation}
the function $f_{a,b}(x)$ has a unique maximum and a unique minimum in sequence;
\item
The necessary condition for the function $f_{a,b}(x)$ to be strictly increasing is
\begin{equation}\label{Carlson-Arccos-thm-1-neces}
(a,b)\in\biggl\{b\ge\frac2\pi-a\biggr\}\cap\biggl\{a\ge\frac12\biggr\}.
\end{equation}
\end{enumerate}
\end{thm}

As direct consequences of the monotonicity of the function $f_{a,b}(x)$, the following inequalities may be deduced.

\begin{thm}\label{Carlson-Arccos-cor-1}
For $x\in(0,1)$, the double inequality
\begin{equation}\label{Carlson-Arccos-ineq-1}
\frac\pi2\cdot\frac{(1-x)^{1/2}}{(1+x)^b}<\arccos x<2^{b+1/2}\cdot\frac{(1-x)^{1/2}}{(1+x)^b}
\end{equation}
holds provided that $b\ge\frac16$.
\par
The right-hand side inequality in~\eqref{Carlson-Arccos-ineq-1} is valid if and only if $b\ge\frac16$.
\par
The reversed version of~\eqref{Carlson-Arccos-ineq-1} is valid provided that $b\le\frac2\pi-\frac12$.
\par
The reversed version of the left-hand side inequality in~\eqref{Carlson-Arccos-ineq-1} is valid if and only if $b\le\frac2\pi-\frac12$.
\par
If $(a,b)$ satisfies~\eqref{Carlson-Arccos-thm-1-a3}, $16ab(b-a)+(a+b)^2>0$ and
\begin{equation}
x_1=\frac{(a+b)(2b-2a+1)-\sqrt{16ab(b-a)+(a+b)^2}\,}{2(a-b)^2}>0,
\end{equation}
then
\begin{multline}
\left.\begin{aligned}
&\min\biggl\{2^{b+1/2},\frac\pi2\biggr\}\frac{(1-x)^a}{(1+x)^b},&a&=\frac12\\
&0,&a&<\frac12
\end{aligned}\right\}
\le\arccos x\\
\le\frac{(1+x_1)^{b+1/2}(1-x_1)^{1/2-a}}{a+b+(a-b)x_1}\cdot\frac{(1-x)^a}{(1+x)^b}.
\end{multline}
\par
If $(a,b)$ satisfies~\eqref{Carlson-Arccos-thm-1-a4}, $16ab(b-a)+(a+b)^2>0$ and
\begin{equation}
x_2=\frac{(a+b)(2b-2a+1)+\sqrt{16ab(b-a)+(a+b)^2}\,}{2(a-b)^2}\in(0,1),
\end{equation}
then
\begin{equation}
\arccos x \ge\frac{(1+x_2)^{b+1/2}(1-x_2)^{1/2-a}}{a+b+(a-b)x_2}\cdot\frac{(1-x)^a}{(1+x)^b}.
\end{equation}
\end{thm}

The second aim of this paper is to sharpen and generalize the left-hand side inequality in~\eqref{Carlson's-ineq-arccos} as follows.

\begin{thm}\label{Carlson-Arccos-cor-2}
For $x\in(0,1)$, the function
\begin{equation}
F_{1/2,1/2,2\sqrt2\,}(x)=\frac{2\sqrt2\,+(1+x)^{1/2}}{(1-x)^{1/2}}\arccos x
\end{equation}
is strictly decreasing. Consequently, the double inequality
\begin{equation}\label{Carlson-Arccos-ineq-2}
\frac{6(1-x)^{1/2}}{2\sqrt2\,+(1+x)^{1/2}}<\arccos x<\frac{\bigl(1/2+\sqrt2\,\bigr)\pi(1-x)^{1/2}}{2\sqrt2\,+(1+x)^{1/2}}
\end{equation}
holds on $(0,1)$ and the constants $6$ and $\bigl(\frac12+\sqrt2\,\bigr)\pi$ in~\eqref{Carlson-Arccos-ineq-2} are the best possible.
\end{thm}

\begin{rem}
From Theorem~\ref{Carlson-Arccos-cor-1}, we obtain the following two double inequalities:
\begin{gather}\label{two-double-Carlson-1}
\frac{\pi(1-x)^{1/2}}{2(1+x)^{1/6}}<\arccos x <\frac{\sqrt[3]4\,(1-x)^{1/2}}{(1+x)^{1/6}}, \quad x\in(0,1);\\
\frac{4^{1/\pi}(1-x)^{1/2}}{(1+x)^{(4-\pi)/2\pi}}<\arccos x <\frac{\pi(1-x)^{1/2}}{2(1+x)^{(4-\pi)/2\pi}}, \quad x\in(0,1).\label{two-double-Carlson-2}
\end{gather}
\par
Except that the right-hand side inequality in~\eqref{two-double-Carlson-1} and the left-hand side inequality in~\eqref{Carlson-Arccos-ineq-2} are same to the corresponding one in~\eqref{Carlson's-ineq-arccos} and that the left-hand side inequality in~\eqref{Carlson's-ineq-arccos} is better than the corresponding one in~\eqref{two-double-Carlson-2}, other corresponding inequalities in~\eqref{Carlson's-ineq-arccos}, \eqref{Carlson-Arccos-ineq-2}, \eqref{two-double-Carlson-1} and~\eqref{two-double-Carlson-2} are not included each other.
\end{rem}

\section{Proofs of theorems}

Now we are in a position to verify our theorems.

\begin{proof}[Proof of Theorem~\ref{Carlson-Arccos-thm-1}]
Straightforward differentiation yields
\begin{align}
\begin{split}\label{f(a,b)(x)}
f_{a,b}'(x)&=\frac{(1+x)^{b-1}}{(1-x)^{a+1}}(\arccos x) \biggl[a+b+(a-b)x-\frac{\sqrt{1-x^2}\,}{\arccos x}\biggr]\\
&\triangleq\frac{(1+x)^{b-1}}{(1-x)^{a+1}}(\arccos x) g_{a,b}(x),
\end{split}\\\notag
g_{a,b}'(x)&=a-b-\frac1{(\arccos x)^2}+\frac{x}{\sqrt{1-x^2}\,\arccos x},\\\notag
g_{a,b}''(x)&=\frac{(\arccos x)^2+x\sqrt{1-x^2}\,\arccos x+2x^2-2}{(1-x^2)^{3/2}(\arccos x)^3}\\\notag
&\triangleq\frac{h(x)}{(1-x^2)^{3/2}(\arccos x)^3},\\\notag
h'(x)&=\frac{(1+2x^2)}{\sqrt{1-x^2}\,}\biggl[{\frac{3x\sqrt{1-x^2}\,}{1+2x^2}-\arccos x}\biggr]\\\notag
&\triangleq\frac{(1+2x^2)}{\sqrt{1-x^2}\,}q(x),\\\notag
q'(x)&=\frac{4(1-x^2)^{3/2}}{(1+2x^2)^2}.
\end{align}
It is clear that $q'(x)$ is positive, and so $q(x)$ is increasing on $[0,1)$. By virtue of $q(1)=0$, we obtain that $q(x)<0$ on $[0,1)$, which equivalent to $h'(x)<0$ and $h(x)$ is decreasing on $[0,1)$. Due to $h(1)=0$, it follows that $h(x)>0$ and $g_{a,b}''(x)>0$ on $[0,1)$, and so the function $g_{a,b}'(x)$ is increasing on $[0,1)$. It is easy to obtain that $\lim_{x\to0^+}g_{a,b}'(x)=a-b-\frac4{\pi^2}$ and $\lim_{x\to1^-}g_{a,b}'(x)=a-b-\frac13$. Hence,
\begin{enumerate}
\item
if $a-b\ge\frac4{\pi^2}$, then $g_{a,b}'(x)>0$ and $g_{a,b}(x)$ is increasing on $(0,1)$;
\item
if $a-b\le\frac13$, then $g_{a,b}'(x)<0$ and $g_{a,b}(x)$ is decreasing on $(0,1)$;
\item
if $\frac13<a-b<\frac4{\pi^2}$, then $g_{a,b}'(x)$ has a unique zero and $g_{a,b}(x)$ has a unique minimum on $(0,1)$.
\end{enumerate}
Direct calculation gives
\begin{equation}\label{g-a,b-0-val}
g_{a,b}(0)=a+b-\frac2\pi
\end{equation}
and
\begin{equation}\label{g-a,b-1-val}
\lim_{x\to1^-}g_{a,b}(x)=2a-1.
\end{equation}
Therefore,
\begin{enumerate}
\item
if $a-b\ge\frac4{\pi^2}$ and $a+b\ge\frac2\pi$, then $g_{a,b}(x)$ and $f'_{a,b}(x)$ are positive, and so the function $f_{a,b}(x)$ is increasing on $(0,1)$;
\item
if $a-b\ge\frac4{\pi^2}$ and $2a\le1$, then $g_{a,b}(x)$ and $f'_{a,b}(x)$ are negative, and so the function $f_{a,b}(x)$ is decreasing on $(0,1)$;
\item
if $a-b\le\frac13$ and $a+b\le\frac2\pi$, then $g_{a,b}(x)$ and $f'_{a,b}(x)$ are negative, and so the function $f_{a,b}(x)$ is decreasing on $(0,1)$;
\item
if $a-b\le\frac13$ and $2a\ge1$, then $g_{a,b}(x)$ and $f'_{a,b}(x)$ are positive, and so the function $f_{a,b}(x)$ is increasing on $(0,1)$;
\item
if $\frac13<a-b<\frac4{\pi^2}$, $a+b\le\frac2\pi$ and $a\le\frac12$, then $g_{a,b}(x)$ and $f'_{a,b}(x)$ are negative, and so the function $f_{a,b}(x)$ is decreasing on $(0,1)$;
\item
if $\frac13<a-b<\frac4{\pi^2}$, $a+b>\frac2\pi$ and $a\le\frac12$, then $g_{a,b}(x)$ and $f'_{a,b}(x)$ have a unique zero on $(0,1)$, which is a unique maximum point of $f_{a,b}(x)$ on $(0,1)$;
\item
if $\frac13<a-b<\frac4{\pi^2}$, $a+b\le\frac2\pi$ and $a>\frac12$, then $g_{a,b}(x)$ and $f'_{a,b}(x)$ have a unique zero on $(0,1)$, which is a unique minimum point of $f_{a,b}(x)$ on $(0,1)$;
\item
if $\frac13<a-b<\frac4{\pi^2}$, the minimum point $x_0\in(0,1)$ of $g_{a,b}(x)$ satisfies
$$
\frac1{\arccos x_0}=\frac{x_0+\sqrt{x_0^2+4(a-b)(1-x_0^2)}}{2\sqrt{1-x_0^2}}
$$
and the minimum of $g_{a,b}(x)$ equals
\begin{align*}
g_{a,b}(x_0)&=a+b+\biggl(a-b-\frac12\biggr)x_0-\frac12\sqrt{x_0^2+4(a-b)(1-x_0^2)}\\
&\ge a+b-\frac{2(a-b)^{3/2}}{\sqrt{4(a-b)-1}\,},
\end{align*}
which means that
\begin{enumerate}
\item
when $\frac13<a-b<\frac4{\pi^2}$ and $a+b\ge\frac{2(a-b)^{3/2}}{\sqrt{4(a-b)-1}\,}$, the functions $g_{a,b}(x)$ and $f'_{a,b}(x)$ are non-negative, and so the function $f_{a,b}(x)$ is strictly increasing on $(0,1)$;
\item
when $\frac13<a-b<\frac4{\pi^2}$, $a+b>\frac2\pi$, $a>\frac12$ and $a+b<\frac{2(a-b)^{3/2}}{\sqrt{4(a-b)-1}\,}$, the functions $g_{a,b}(x)$ and $f'_{a,b}(x)$ have two zeros which are in sequence the maximum and minimum of the function $f_{a,b}(x)$ on $(0,1)$.
\end{enumerate}
\end{enumerate}
As a result, the sufficiency for the function $f_{a,b}(x)$ to be monotonic on $(0,1)$ is proved.
\par
Conversely, if the function $f_{a,b}(x)$ is strictly decreasing, then the function $g_{a,b}(x)$ must be negative on $(0,1)$, so the quantities in~\eqref{g-a,b-0-val} and~\eqref{g-a,b-1-val} are non-positive. Hence, the condition in~\eqref{Carlson-Arccos-thm-1-a2} is also necessary.
\par
By the similar argument as above, the necessary condition~\eqref{Carlson-Arccos-thm-1-neces} follows immediately. The proof of Theorem~\ref{Carlson-Arccos-thm-1} is thus proved.
\end{proof}

\begin{proof}[Proof of Theorem~\ref{Carlson-Arccos-cor-1}]
It is easy to see that
$$
\lim_{x\to0^+}f_{a,b}(x)=\frac\pi2
$$
and
$$
\lim_{x\to1^-}f_{a,b}(x)=2^b\lim_{x\to1^-}\frac{\arccos x}{(1-x)^a} =
\begin{cases}
2^{b+1/2},&a=\frac12;\\
0,&a<\frac12;\\
\infty,&a>\frac12.
\end{cases}
$$
From Theorem~\ref{Carlson-Arccos-thm-1}, it follows that the function $f_{1/2,b}(x)$ is strictly increasing (or strictly decreasing respectively) on $(0,1)$ if $b\ge\frac16$ (or if and only if $b\le\frac2\pi-\frac12$ respectively). Consequently, if $b\ge\frac16$, then
\begin{equation}\label{frac-pi2}
\frac\pi2=\lim_{x\to0^+}f_{1/2,b}(x)<f_{1/2,b}(x)<\lim_{x\to1^-}f_{1/2,b}(x)=2^{b+1/2}
\end{equation}
on $(0,1)$, which can be rearranged as the inequality~\eqref{Carlson-Arccos-ineq-1}; if $b\le\frac2\pi-\frac12$, the inequality~\eqref{frac-pi2}, and so the inequality~\eqref{Carlson-Arccos-ineq-1}, reverses.
\par
The right-hand side inequality in~\eqref{Carlson-Arccos-ineq-1} may be rewritten as
\begin{align*}
b&>\frac{\ln\arccos x-\frac12\ln(1-x)-\frac12\ln2}{\ln2-\ln(1+x)}\\
&\to (1+x)\biggl[\frac1{\sqrt{1-x^2}\,\arccos x}-\frac1{2(1-x)}\biggr]\\
&\to\frac{2(x-1)+\sqrt{1-x^2}\,\arccos x}{(x-1)\sqrt{1-x^2}\,\arccos x} \\
&\to\frac{x\arccos x/\sqrt{1-x^2}\,-1}{(x-1)\bigl[1+(1+2x)\arccos x/\sqrt{1-x^2}\,\bigr]}\\
&\to\frac{x\arccos x/\sqrt{1-x^2}\,-1}{4(x-1)}\\
&\to\frac16
\end{align*}
as $x\to1^-$. Therefore, the condition $b\ge\frac16$ is also a necessary condition such that the right-hand side inequality in~\eqref{Carlson-Arccos-ineq-1} is valid.
\par
The reversed version of the left-hand side inequality in~\eqref{Carlson-Arccos-ineq-1} may be rearranged as
$$
b<\frac{\ln\pi-\ln2+[\ln(1-x)]/2-\ln\arccos x}{\ln(1+x)}\to\frac2\pi-\frac12
$$
as $x\to0^+$. Hence, the necessity of $\frac2\pi-\frac12$ is proved.
\par
By the equation~\eqref{f(a,b)(x)} in the proof of Theorem~\ref{Carlson-Arccos-thm-1}, it follows that the extreme points $\xi\in(0,1)$ of the function $f_{a,b}(x)$ satisfy $g_{a,b}(\xi)=0$, that is,
$$
\arccos \xi=\frac{\sqrt{1-\xi^2}\,}{a+b+(a-b)\xi},
$$
so the extremes of $f_{a,b}(x)$ equal
$$
f_{a,b}(\xi)=\frac{(1+\xi)^{b+1/2}}{(1-\xi)^{a-1/2}[a+b+(a-b)\xi]}\triangleq g(\xi)
$$
and
$$
g'(x)\triangleq\frac{(x+1)^{b-1/2}h(x)}{[a+b+(a-b)x]^2(1-x)^{a+1/2}},
$$
where
$$
h(x)=\bigl(a-b)^2x^2+(a+b)(2a-2b-1)x+(a+b)^2-a+b
$$
has two zero points $x_1$ and $x_2$ which are also the zeros of the function $g'(x)$ and the extreme points of $g(x)$ for $x\in(0,1)$.
\par
When $16ab(b-a)+(a+b)^2>0$ and $x_{1,2}\in(0,1)$, the point $x_1$ is the maximum point and $x_2$ is the minimum point of $g(x)$, so we have the inequality
\begin{equation}\label{max-min-ineq-carlson}
\frac{(1+x_2)^{b+1/2}(1-x_2)^{1/2-a}}{a+b+(a-b)x_2}\le f_{a,b}(\xi) \le\frac{(1+x_1)^{b+1/2}(1-x_1)^{1/2-a}}{a+b+(a-b)x_1}.
\end{equation}
\par
When $16ab(b-a)+(a+b)^2>0$ such that $x_1\le0$ and $x_2\in(0,1)$, the function $g(x)$ has only one minimum and the left-hand side inequality in~\eqref{max-min-ineq-carlson} is valid.
\par
When $16ab(b-a)+(a+b)^2>0$ such that $x_1\in(0,1)$ and $x_2\ge1$, the function $g(x)$ has only one maximum and the right-hand side inequality in~\eqref{max-min-ineq-carlson} is valid.
\par
When $16ab(b-a)+(a+b)^2>0$ such that $x_{2}\le0$ or $x_{1}\ge1$, the function $g(x)$ is strictly increasing on $(0,1)$; since
$$
\lim_{x\to0^+}g(x)=\frac1{a+b}\quad\text{and}\quad \lim_{x\to1^-}g(x)=
\begin{cases}
2^{b+1/2},&a=\frac12,\\
0,&a<\frac12,\\
\infty,&a>\frac12,
\end{cases}
$$
we have
\begin{equation}\label{max-min-ineq-carlson-2}
\frac1{a+b}\le f_{a,b}(\xi) \le
\begin{cases}
2^{b+1/2},&a=\frac12,\\
0,&a<\frac12,\\
\infty,&a>\frac12.
\end{cases}
\end{equation}
\par
When $16ab(b-a)+(a+b)^2>0$ such that $x_1\le0$ and $x_2\ge1$, the function $g(x)$ is strictly decreasing on $(0,1)$, and so the inequality~\eqref{max-min-ineq-carlson-2} reverses.
\par
When $16ab(b-a)+(a+b)^2\le0$, the function $g(x)$ is strictly increasing on $(0,1)$, and so the inequality~\eqref{max-min-ineq-carlson-2} holds.
\par
Under the condition~\eqref{Carlson-Arccos-thm-1-a3},
\begin{enumerate}
\item
if $16ab(b-a)+(a+b)^2>0$ and $x_1>0$, then
\begin{multline*}
\left.\begin{aligned}
&\min\biggl\{2^{b+1/2},\frac\pi2\biggr\},&a&=\frac12\\
&0,&a&<\frac12
\end{aligned}\right\}
\le\frac{(1+x)^b}{(1-x)^a}\arccos x\\
\le f_{a,b}(\xi) \le\frac{(1+x_1)^{b+1/2}(1-x_1)^{1/2-a}}{a+b+(a-b)x_1};
\end{multline*}
\item
if either $16ab(b-a)+(a+b)^2>0$ such that $x_{2}\le0$ or $x_{1}\ge1$ or $16ab(b-a)+(a+b)^2\le0$, then
\begin{multline*}
\left.\begin{aligned}
&\min\biggl\{2^{b+1/2},\frac\pi2\biggr\},&a&=\frac12\\
&0,&a&<\frac12
\end{aligned}\right\}
\le\frac{(1+x)^b}{(1-x)^a}\arccos x\\
\le f_{a,b}(\xi) \le
\begin{cases}
2^{b+1/2},&a=\dfrac12,\\[0.6em]
0,&a<\dfrac12.
\end{cases}
\end{multline*}
\end{enumerate}
\par
Under the condition~\eqref{Carlson-Arccos-thm-1-a4}, if $16ab(b-a)+(a+b)^2>0$ and $x_2\in(0,1)$, then
\begin{equation*}
\frac{(1+x)^b}{(1-x)^a}\arccos x
\ge f_{a,b}(\xi) \ge\frac{(1+x_2)^{b+1/2}(1-x_2)^{1/2-a}}{a+b+(a-b)x_2}.
\end{equation*}
Straightforward simplification completes the proof of Theorem~\ref{Carlson-Arccos-cor-1}.
\end{proof}

\begin{proof}[Proof of Theorem~\ref{Carlson-Arccos-cor-2}]
Direct computation yields
\begin{align*}
\frac{\td F_{1/2,1/2,2\sqrt2\,}(x)}{\td x}&= \frac{\bigl[1+\sqrt{2(x+1)}\,\bigr]\sqrt{1-x^2}\,}{(1+x)(x-1)^2}\\
&\quad\times\Biggl[\arccos x -\frac{\bigl(\sqrt{1+x}\,+2\sqrt{2}\,\bigr)\sqrt{1-x}\,}{1+\sqrt{2(x+1)}\,}\Biggr]\\
&\triangleq\frac{\bigl[1+\sqrt{2(x+1)}\,\bigr]\sqrt{1-x^2}\,}{(1+x)(x-1)^2}G(x),\\
G'(x)&=\frac{(x-1)\sqrt{2(1+x)}\,\bigl[\sqrt{1+x}\,-\sqrt{2}\,\bigr]} {2\sqrt{(1+x)(1-x^2)}\,\bigl[1+\sqrt{2(1+x)}\,\bigr]^2}\\
&>0.
\end{align*}
Thus, the function $G(x)$ is strictly increasing on $(0,1)$. Since $\lim_{x\to1^-}G(x)=0$, the function $G(x)$ is negative on $(0,1)$, which means that the derivative $F_{1/2,1/2,2\sqrt2\,}'(x)$ is negative and that the function $F_{1/2,1/2,2\sqrt2\,}(x)$ is strictly decreasing on $(0,1)$. Further, from
$$
\lim_{x\to0^+}F_{1/2,1/2,2\sqrt2\,}(x)=\biggl(\frac12+\sqrt2\,\biggr)\pi \quad\text{and}\quad \lim_{x\to1^-}F_{1/2,1/2,2\sqrt2\,}(x)=6,
$$
the double inequality~\eqref{Carlson-Arccos-ineq-2} and its best possibility follow.
\end{proof}

\end{document}